\documentclass{amsart}
\usepackage[utf8]{inputenc}
\usepackage{geometry}
\usepackage{graphicx}
\graphicspath{ {./images/} }
\usepackage[dvipsnames]{xcolor}
\usepackage[utf8]{inputenc}
\usepackage{cite}
\usepackage{geometry, amsthm, amsmath, amssymb, graphicx, float, enumerate, hyperref, longtable, braket, mathtools, bbm}
\usepackage [english]{babel}
\usepackage{amsmath}

\usepackage{autonum}
\usepackage{enumitem}

\DeclareMathAlphabet{\mathpzc}{OT1}{pzc}{m}{it}

\restylefloat{table}
\restylefloat{figure} 

\newtheorem{theorem}{Theorem}

\newtheorem{corollary}[theorem]{Corollary}

\newtheorem{lemma}[theorem]{Lemma}
\newtheorem{proposition}[theorem]{Proposition}
\theoremstyle{definition}
\newtheorem{definition}[theorem]{Definition}

\theoremstyle{remark}

\newtheorem*{corollary-non}{Corollary}

\newcommand{\dottedcolumn}[3]{
	\settowidth{\dimen0}{$#1$}
	\settowidth{\dimen2}{$#2$}
	\ifdim\dimen2>\dimen0 \dimen0=\dimen2 \fi
	\begin{pmatrix}\,
		\vcenter{
			\kern.6ex
			\vbox to \dimexpr#1\normalbaselineskip-1.2ex{
				\hbox{$#2$}
				\kern3pt
				\xleaders\vbox{\hbox to \dimen0{\hss.\hss}\vskip4pt}\vfill
				\kern1pt
				\hbox{$#3$}
			}\kern.6ex}\,
	\end{pmatrix}
}
\DeclarePairedDelimiter{\norm}{\lVert}{\rVert}

\title{Deconvolution on graphs via linear programming}
\author{Bernhard G. Bodmann}\thanks{This research was supported in part by NSF Grant DMS ATD-1925352}
\address{641 Philip G. Hoffman Hall, Department of Mathematics, University of Houston, Houston, TX 77204-3008} \email{bgb@math.uh.edu} 
\author{Jennifer J. May}
\address{641 Philip G. Hoffman Hall, Department of Mathematics, University of Houston, Houston, TX 77204-3008}
\email{jrmay@uh.edu}

\begin{document}

\begin{abstract}
	The main challenge addressed in this paper is to identify individual terms in a superposition
of heat kernels on a graph. We establish geometric conditions on the vertices at which these 
heat kernels are centered and find bounds on the time parameter governing the evolution under the heat semigroup
that guarantee a successful recovery. 
This result can be viewed as a type of deconvolution on a graph. A 
first main result addresses the setting of a common time parameter for all the heat kernels. We also
treat a more general setting when the time parameter depends on the location at which the heat kernel is centered.  
\end{abstract}

\maketitle


\section{Introduction}

Sparse recovery is generally understood as the identification of individual terms in a linear combination 
of vectors from a dictionary. A main challenge is to establish an accurate estimate of these terms
even if one observes only a noisy version of a low-dimensional projection of the given linear combination. 
Under the theme of ``compressed sensing off the grid'' or ``superresolution'', this type of task has been
accomplished for identifying superpositions of sinusoids \cite{Tang} or finitely supported measures from lowest Fourier coefficients \cite{Candgrand}. In yet another work, these ideas were developed in the setting of Hilbert spaces
with reproducing kernels \cite{BodmannKutyniokFlinth}
or in spline-type spaces \cite{Unser:2021}. In the present paper, 
we choose an application to graph signal processing \cite{MR3038378,MR3225164}, where
ideas from harmonic analysis are realized in a less structured setting given by discrete geometries.
In a few papers, the idea of sampling a function and finding an accurate approximation has been developed,
mostly for a class of smooth functions, see \cite{MR2737764,MR3385544,MR3747603,MR3841742}.
Here, we focus on the recovery of a sparse function, whose support is a small fraction of the graph, when only a smoothed
version of the function is  observed. The smoothing replaces any function supported in one point by a heat kernel 
centered at this point, with a given time parameter. The time parameter can be fixed, independent of 
the point of support, or vary across the graph.   
A main insight in this paper is that dual certificates or approximate dual certificates are helpful
to establish a domain for the time parameter and a range of support patterns that permit sparse recovery. 
 
For the first main task addressed in this paper, we assume that a function $g$
has support $\{v_j\}_{j=1}^J$ and the 
values $g(v_j)=c_j$, $j \in \{1, 2, \dots, J\}$.
We wish to recover an approximation to $g$ from
observing
the smoothed version $e^{t \Delta} g$ corrupted by noise,
$$
	f=\sum_{j=1}^{J} c_j e^{t\Delta}\delta_{v_j} + \omega
$$
where $t \ge 0$ is the time parameter for the heat semigroup $\{e^{t\Delta}\}_{t \ge 0}$, for each $x \in V$, $\delta_x$ is the function whose only non-zero value 
is $\delta_x(x)=1$, and
$\omega$ is the additive noise. Applying a linear inversion of $e^{t \Delta}$ is problematic. 
Even with a regularization step using an orthogonal projection $P_\Lambda$, which projects on the sum 
of all eigenspaces corresponding to eigenvalues of the negative Laplacian $-\Delta$ that are in the interval $[0,\Lambda]$,
the approximately recovered vector is
$$
    \hat g = e^{-t \Delta} P_\Lambda f =  \sum_{j=1}^{J} c_j P_\Lambda \delta_{v_j} + P_\Lambda e^{-t \Delta} \omega
$$
and the worst-case error norm is bounded below by $ |e^{t \Lambda}\|\omega\|_2 - \| (I-P_\Lambda) g\|_2 |$, which exhibits exponential growth in $t$. A linear recovery procedure based on regularizing with $P_\Lambda$ and
inverting the heat semigroup then needs to trade off errors between the effect of $P_\Lambda$ on $g$ and the effect
of the noise term $P_\Lambda \omega$ when the heat semigroup is inverted.  

In this paper, we establish a simple procedure that only requires an a priori knowledge of the 
noise norm $\|\omega\|_2$ in order to find an accurate approximation of the original sparse signal.
To this end, we formulate a linear program that minimizes the $\ell^1$-norm among all signals
that are consistent with the observed data. This strategy has an extensive history in applications
of compressed sensing \cite{CandesEmTao,errorcantao,Foucart:2013vn}.  In the last part of the paper, we
adapt the recovery strategy to observing
$$
   f=\sum_{j=1}^{J} c_j e^{t_j\Delta}\delta_{v_j} + \omega \, ,
$$
where the global time parameter $t$ has been replaced by $\{t_j\}_{j=1}^J$.


\begin{definition} \label{LP}
Let $(V,E)$ be a finite, simple graph with Laplacian $\Delta$, $f\in \ell^2(V)$ 
and $\epsilon, t \ge 0$.
The vector $\hat g \in \ell^2(V)$ is a solution of the linear program (LP) associated with the heat semigroup if $\|\hat g \|_1 \le \|\tilde g\|_1$
for each $\tilde g$ with $\|e^{t \Delta} \tilde g - f \|_2 \le  \epsilon$.
\end{definition}

This definition is chosen so that when $f=e^{t\Delta} g +\omega$ with $g,\omega \in \ell^2(V)$ unknown and $\|\omega\|_2 \le \epsilon$, then $g$ is in the feasible domain over which the $\ell^1$-norm is minimized.

We use dual certificates, elements in $\ell^\infty(V)$, to verify that a vector is a norm minimizer. To illustrate this, we first consider the
noiseless case.

\begin{definition}
If $f, \tilde g \in \ell^2(V)$ satisfy $e^{t \Delta}  \tilde g = f$ and
there exists $a \in \ell^2(V)$ such that $h = e^{t\Delta} a$ gives $\|h\|_\infty \le 1$ and
$\|\tilde g\|_1 = \langle f, a\rangle$, then we say that $h$ is a dual certificate for $\tilde g$.
\end{definition}

In this case, $\tilde g$ is a solution of (LP) (see Definition \ref{LP}) because
 $\|\tilde g\|_1= \langle f, a \rangle = \langle g, h\rangle = |\langle g, h \rangle |=\|g\|_1$ rules out that there is
 another feasible choice for $\tilde g$ with a lower norm.
Further below, we recall that the same argument can also be used to show uniqueness and 
stability with respect to noise, when the dual certificate satisfies additional conditions.

In our context, 
for $g = \sum_{j=1}^J c_j \delta_{v_j}$, we construct a dual certificate
in the form
$$
   h(x) = \sum_{j=1}^J a_j e^{t \Delta} \delta_{v_j}(x) \, .
$$

In the absence of noise, the conditions then imply $\|h\|_\infty = 1$ and $h(v_j) = c_j /|c_j|$ for each
$j \in \{1, 2, \dots, J\}$.
We argue qualitatively why such a dual certificate exists for small time. If $t=0$, then
setting $h(x) = c_j /|c_j|$ if $x=v_j$ and otherwise $h(x)=0$ gives the certificate.
We denote $S=\{v_j\}_{j=1}^J$.
The heat semigroup $\{e^{t\Delta}\}_{t \ge 0}$ is invertible on $\ell^2(V)$, hence 
also on the subspace $\ell^2(S)$. We denote this operator by 
$M^t$, and its inverse by $M^{-t}$.
To be precise, if $f \in \ell^2(S)$, then we extend $f$ by $f(x) = 0$ if $x \not\in S$
and let for $y \in S$
$$
    (M^t f)(y) =  (e^{t \Delta} f) (y) \, .
$$ 
By continuity, 
$\| M^{-t} h\|_\infty$ is close to $\|h\|_\infty = 1$ for all sufficiently small $t$. 
Now extending $a\equiv M^{-t} h$
by zero on $V\setminus S$ and using that $e^{t \Delta} a$ is continuous in the time parameter, for sufficiently small $t$,
$|(e^{t\Delta} a)(x)|<1$ for each $x \not \in S$. The main portion of this paper makes this argument 
quantitative, with explicit conditions for the geometry of the set $S$ and bounds on the time of the heat
semigroup that permit the construction of a dual certificate. The main resource for the results presented here
is controlling the decay of the heat kernel on the graph. With bounds on the heat kernel and the discrete geometry of the graph, we have an alternative to elements of harmonic analysis that have been used to construct dual certificates for similar results in the Euclidean setting
\cite{Candgrand}.

The remainder of this paper is organized as follows: In Section~\ref{sec:prelim}, we fix notation.
In the interest of keeping the exposition accessible, we first treat the noiseless case in Section~\ref{sec:noiseless},
and then build the case for noisy recovery in Section~\ref{sec:noisyrecovery}. 
The last section considers the more general case of a spatially dependent time parameter for the heat kernels.

\section{Preliminaries}\label{sec:prelim}

Much of the material under consideration is based on the concept of a weighted graph
and an associated graph Laplacian.

\begin{definition}
An edge-weighted finite graph is a triple $(V,E,b)$ where $V$ is the finite set of vertices,
the edge set $E$ contains subsets of $V$ of size two, and $b: V\times V \to [0,\infty)$ 
has the property $b(x,y) = b(y,x)>0$ if and only if $\{x,y \} \in E$.
We say that $d: V \times V \to [0,\infty)$ is a distance function that is compatible with 
the weight $b$ if it is a metric on $V$ such that for $x , y \in V$, $d(x,y) = \min_{z \in V} ({d(x,z)+d(z,y)})$, and
for each $x \in V$, 
$$ \sum_{y: \{x,y\} \in E}b(x,y) (d(x,y))^2 \le 1 \, . $$
\end{definition}

This type of distance has been investigated by Davies in the study of semigroups generated by graph Laplacians \cite{DaviesII}. 
If we assign to each edge $\{x,y\}$ the weight  $b(x,y) = (d(x,y))^{-2} / k$ with $k$ the maximal degree of the graph,
then the definition of a graph Laplacian is given by a difference quotient with similarities to Euclidean geometry.

\begin{definition} 
	The valency of a vertex $x \in V$ is the sum of all the weights on the edges that are incident to $x$.  We denote it as
	$$
	\text{val}(x) = \sum_{y: \{x,y\} \in E} b(x,y) \, .
	$$
	We define $\displaystyle{\kappa:=\norm{\text{val}}_{\infty}=\max_{x\in V}\{\text{val}(x)\}.}$
\end{definition}

\begin{definition}
Let $(V,E,b)$ be an edge-weighted finite graph, then the associated graph Laplacian $\Delta$ is defined in
terms of the quadratic form $f \mapsto \langle \Delta f, f\rangle$ on $\ell^2(V)$,
$$
   \langle \Delta f, f\rangle = - \frac 1 2 \sum_{\{x,y\} \in E} b(x,y) |f(x)-f(y)|^2 \, .
$$
\end{definition}

The operator $\Delta$ is Hermitian and hence we can define the associated heat semigroup 
$\{e^{t\Delta}\}_{t \ge 0}$, either through the spectral representation or by the power series
expansion of the exponential function. We also describe this semigroup in terms of its kernel function,
$$
    K(t,x,y) \equiv (e^{t \Delta} \delta_y) (x) \, .
$$
The main problem we investigate is under which conditions we can identify $g$ when observing $f=e^{t\Delta} g + \omega$.
We consider a class of signals whose support size is bounded by $J$ and introduce a quantitative notion of separation
for the support.
 \begin{definition} Given a metric $d$ on $V$,
 a set $S \subset V$ satisfies $S \in \mathcal{G}_{J,D}$ if and only if
 $ |S| \le J$ and if $\{x,y \} \subset S$ then $d(x,y) \ge D$.
 We denote the set of functions whose support is in $\mathcal{G}_{J,D}$ by
 $\Gamma_{J,D}$.
 \end{definition}
This paper will establish approximate inverses for the heat equation in both the noiseless and the noisy cases.
%
\section{Recovery in the Noiseless Case}  \label{sec:noiseless}
We claim for small $t$ and a minimum distance between support vertices that we can create a dual certificate that facilitates recovering the signal $g$ from the observed signal $f$. The following  result can be found in Cand{\`e}s
and Fernandez-Granda \cite{Candgrand}, as well as Tang \cite{Tang}.

\begin{theorem}\label{thm:dualcert}
Let $g \in \ell^2(V)$, $g = \sum_{v \in S} c_j \delta_{v_j}$ and support $S \subsetneq V$.
Suppose there exists $a \in \ell^2(V)$ such that $e^{t\Delta}a=:h$ has the following properties:
\begin{enumerate}
	\item $h(v_j)=c_j/|c_j|$ for all $v_j \in S$
	\item $\vert h(x)\vert <1$ for all $x \in V \backslash S$,
\end{enumerate}
then $g$ is the unique solution to the $\ell^1$-minimization problem (LP) (as in Definition \ref{LP})
satisfying $e^{t\Delta}g=f.$
\end{theorem}

\begin{proof}
Assume such a $h$ exists and let $\tilde{g}$ be a solution to the minimization problem described.  Then, $\tilde{g}$ is such that $e^{t\Delta}\tilde{g}=e^{t\Delta}g.$  Moreover,\\
\begin{equation}
	\begin{split}
		\langle \tilde{g},h \rangle & = \langle e^{t\Delta}\tilde{g},a \rangle \\
		&= \langle g, e^{t\Delta}a \rangle \\
		&= \langle g,h \rangle.\\
	\end{split}
\end{equation}
Now, assuming $\tilde{g} \neq g$ then $\tilde{g}=g+w$, with $w \ne 0$ and 
$\langle w,h \rangle = 0$. We can now decompose $w$ into a part that is zero outside of the support of $g$ and 
the remainder, $w = w\chi_S + w\chi_{S^c}$.  Here, for any set $Q\subset V$, $\chi_Q$ is the characteristic function of $Q$, 
$\chi_Q = \sum_{v \in Q} \delta_v$.
From $e^{t \Delta} \tilde g= f$, $$0=\langle h,w \rangle = \langle h, w\chi_S \rangle + \langle h, w\chi_{S^c} \rangle \, ,$$
so $$ \vert \langle h, w\chi_S \rangle\vert = \vert \langle h, w\chi_{S^c} \rangle \vert$$
and hence
$$\norm{w\chi_S}_1 = \vert \langle h, w\chi_{S^c}\rangle \vert \leq \norm{w\chi_{S^c}}_1 \cdot \norm{h \chi_{S^c}}_{\infty}.$$
Using property (2) for $h$ then $$\norm{w\chi_S}_1 < \norm{w\chi_{S^c}}_1.$$

Now considering
\begin{equation}
	\begin{split}
		\norm{g}_1 & \geq \norm{\tilde{g}}_1 = \norm{g+w}_1\\
		& = \norm{g+w\chi_S+w\chi_{S^c}}_1\\
		& = \norm{g + w\chi_S}_1 + \norm{w\chi_{S^c}}_1\\
		& \geq \norm{g}_1 - \norm{w\chi_S}_1 + \norm{w\chi_{S^c}}_1\\
& > 
\norm{g}_1\\
\end{split}
\end{equation}
gives a contradiction and hence $w=0$, proving the solution is unique.
\end{proof}

The next step is to establish the existence of such a dual certificate.
We recall that we assume a distance function on the graph that is compatible with the weight
in the definition of the graph Laplacian.

To prepare the construction of the dual certificate, we
derive a lower bound for the diagonal entries of the heat kernel based on spectral properties of the Laplacian.

	\begin{lemma}\label{lower}
	Suppose  $G=(V,E,b)$ is an edge-weighted finite connected graph that has $N$ vertices and maximum valency   $\kappa=\norm{\text{val}}_{\infty}.$  Let $\Delta$ be the graph Laplacian of $G$
	with  $\lambda$ the smallest non-zero eigenvalue of $-\Delta$.	 
	Then for $x \in V$,
	$$ \frac{1}{N}+e^{-t\left(\frac{N}{N-1}\right)\kappa}
	\left(\frac{N-1}{N}\right)
	\le  K\left(t,x,x\right) \le  \frac{1}{N}+e^{-t\lambda}\left(\frac{N-1}{N}\right) \,. $$
	
\end{lemma}

\begin{proof} We recall that the constant function is in the eigenspace of $\Delta$ corresponding to 
	eigenvalue zero. We can split $\delta_x$ in a linear combination of a constant and a zero-summing vector,
	$$
	   \delta_x = \nu_x + \mu
	$$
	where $\nu_x = \delta_x - \mu$, and the constant function $\mu = \frac 1 N$ satisfies $\Delta \mu = 0$.
	We compute the (squared) norms,
	yielding $$\|\mu\|_2^2 = \frac 1 N  \text{~ and ~~}  \norm{\nu_x}_2^2=1-\frac{1}{N}=\frac{N-1}{N} \, .$$

	Next, we consider the diagonal entries of the heat semigroup:
	\begin{equation}
		\begin{split}
			K(t,x,x) &= \langle e^{t\Delta}\delta_x,\delta_x\rangle = \langle e^{t\Delta}(\mu +\delta_x-\mu), \mu + \delta_x-\mu \rangle \\
			&= \langle e^{t\Delta}\mu, \mu \rangle + \langle e^{t\Delta}\mu, \nu_x \rangle+ \langle  e^{t\Delta}\nu_x, \mu \rangle + \langle e^{t\Delta}\nu_x, \nu_x \rangle \\
			&= \norm{\mu}_2^2 + \langle e^{t\Delta}\nu_x, \nu_x \rangle \\
			&= \frac{1}{N} + \frac{\langle e^{t\Delta}\nu_x, \nu_x \rangle }{\norm{\nu_x}_2^2}\norm{\nu_x}_2^2\\
			&\geq \frac{1}{N}+\left(\frac{N-1}{N}\right)e^{\left(t\langle \Delta \nu_x, \nu_x\rangle/\norm{\nu_x}_2^2\right)}\quad \text{ (by Jensen's inequality for operators)} \\
			&= \frac{1}{N}+\left(\frac{N-1}{N}\right)e^{\left(-t\left(\frac{N}{N-1}\right)\text{val}(x)\right)} \\
			&\geq \frac{1}{N}+\left(\frac{N-1}{N}\right)e^{-t\left(\frac{N}{N-1}\right)\kappa}.
		\end{split}
	\end{equation}
	A similar line of reasoning gives
	$$
	K(t,x,x) \le \frac{1}{N}+\left(\frac{N-1}{N}\right) e^{-t\lambda} \, ,
	$$
	where we use $\langle e^{t \Delta} \nu_x, \nu_x\rangle \le e^{-t \lambda} \norm{\nu_x}_2^2$
	because $\nu_x$ is orthogonal to $\mu$. 
\end{proof}


We rely on a work by Folz \cite{Folz}, building on Davies \cite{Davies} to bound the off diagonal entries of the heat kernel matrix.

\begin{lemma}[Theorem 2.1 \cite{Folz}]\label{dave}
	Let $(V,E,b)$ be an edge-weighted finite graph and $d$ a compatible distance function,
	then the heat kernel has the bound for each $x, y \in V$ and $t>0$,
	  $$K(t,x,y) \leq \exp\Big[-\frac 1 2 d(x,y)\log\left(\frac{d(x,y)}{2et}\right)\Big] \, .$$
	\end{lemma}

The bounds on the heat kernel establish invertibility of the operator on $\ell^2(S)$, if the vertices
in $S$ are sufficiently separated. 

\begin{proposition}\label{inverse}
	Given an edge-weighted finite connected graph $G=(V,E,b)$ with a compatible distance function $d$, maximum valency    $\kappa:=\norm{\text{val}}_{\infty}$ ,
	the graph Laplacian $\Delta$, and $|V|=N$. Let $S$ have the separation of elements requirement, such that $S \in \mathcal{G}_{J,D}$.  Consider the operator $M^t:=e^{t\Delta}$  on $\ell^2(S)$.
	For $t \le T$, if
	$$(J-1)\left(\frac{2eT}{D}\right)^{D/2}<\frac{1}{N}+e^{-T\left(\frac{N}{N-1}\right)\kappa}
	\left(\frac{N-1}{N}\right),$$
	then $M^t$ is invertible on  $\ell^2( S)$ and the operator norm of the inverse is bounded by
	\begin{equation}\label{Mtinversebounds}
		\|M^{-t}\| \le \Bigl( \frac{1}{N}+e^{-T\left(\frac{N}{N-1}\right)\kappa}
		\left(\frac{N-1}{N}\right)
		- (J-1)\left(\frac{2eT}{D}\right)^{D/2} \Bigr)^{-1} \, .
	\end{equation} 
	
\end{proposition}
\begin{proof}
	We note that the assumed inequality implies that $D>2eT$, otherwise this contradicts with the right-hand side
	being bounded above by one. Thus, we can use monotonicity in $d(x,y) \ge D$ to further estimate Folz's result (Lemma \ref{dave}),
	$$
	K(t,x,y)\leq \exp\left[-\frac 1 2 D\ln\left(\frac{D}{2et}\right)\right]
	= \left( \frac{2et}{D} \right)^{D/2},
	$$
	and then use monotonicity in $t$ to bound off-diagonal elements of $M^t$ for each $t \le T$.  That is, 
	$$
	K(t,x,y) \le \left( \frac{2et}{D} \right)^{D/2} \le \left( \frac{2eT}{D} \right)^{D/2}.
	$$
	Now using the Levy-Desplanques Theorem, a corollary to the Ger\u{s}gorin's theorem, (see \cite{HornJohnson}, Corollary 5.6.17), $M^t$ is invertible because
	it is symmetric and each eigenvalue is bounded below by 
	the minimum among the diagonal entries minus the largest sum occurring among the $(J-1)$ off-diagonal entries in each row. Consequently, the operator norm of the inverse is bounded by the inverse of the lower bound for the 
	smallest eigenvalue of $M^t$.
\end{proof} 

Now that we have shown the heat semigroup is invertible on the vertices in the support of our signal, we look towards including the non-support vertices.  To do so, we will use a classical result from Varah \cite{VARAH},
which controls the $\ell^\infty$ norm under $M^{-t}$.

\begin{lemma}
\label{varah}
Given $G$, $\Delta$, $S \in \mathcal{G}_{J,D}$ and $M^t$ as in Proposition \ref{inverse}, with $t\le T$, then if $h\in \ell^{\infty}(S)$ satisfies $h(v_j)=c_j/|c_j|$ for $j \in \{1, 2, \dots, J\} $
and $a = M^{-t} h$, we have 
$$
\|a\|_\infty  \le \left(\frac{1}{N}+e^{-T\left(\frac{N}{N-1}\right)\kappa}
\left(\frac{N-1}{N}\right) - (J-1) \left(\frac{2eT}{D}\right)^{D/2} \right)^{-1}.
$$
\end{lemma}
\begin{proof} 
If a matrix $A=(a_{k,j})_{k,j=1}^J$ is strictly diagonally dominant matrix, then the operator norm of its inverse on $\ell^\infty$
is bounded by 
$$\norm {A^{-1}}_{\infty,\infty} \le \left(\min_{k}\left(\lvert a_{kk} \rvert - \sum_{j \neq k} \lvert a_{kj} \rvert \right)\right)^{-1}.$$
Now inserting the heat semigroup $M^t$ for $A$ gives 
in our case  that if $|h| \vert\chi_S = 1$, then $a= M^{-t} h$ has 
$$\|a\|_\infty \le 
\bigl(\min_x (K(t,x,x) - (J-1)\sum_{y: y \ne x} K(t,x,y) \bigr)^{-1} \, .$$
With Folz's bound, this yields again the same explicit expression as for the $\ell^2$-norm,
$$
   \|a\|_\infty  \le \left(\frac{1}{N}+e^{-T\left(\frac{N}{N-1}\right)\kappa}
   \left(\frac{N-1}{N}\right) - (J-1) \left(\frac{2eT}{D}\right)^{D/2} \right)^{-1}.
$$
\end{proof}

\begin{theorem}\label{main}
	Suppose we have a finite edge-weighted connected graph, $G=(V,E,b)$ with a compatible distance function, $d$ and maximum valency,    $\kappa:=\norm{\text{val}}_{\infty}.$ Let $\Delta$ be the graph Laplacian of $G$,  and $N = |V|$ Given a subset of vertices of size $J$ with a minimum distance $D$ between any two vertices, that is, $S\in \mathcal{G}_{J,D},$
	and let $\zeta$ be the smallest distance between vertices in $V$.

	If the following relation hold for  $J$, $T \ge t$, and $D$, 
\begin{equation}
    (J-1)\left(\left(\frac{4eT}{D}\right)^{D/4}\left(1+\left(\frac{eT}{D}\right)^{D/4}\right)\right)<\frac{1}{N}+e^{-T\left(\frac{N}{N-1}\right)\kappa}\left(\frac{N-1}{N}\right)-\left(\frac{2eT}{\zeta}\right)^{\zeta/2}
\end{equation}
	then for any choice $\epsilon_j \in \{\pm 1\}$,  $j \in \{1, 2, \dots, J\}$, there is a linear combination, $h(x)=\displaystyle\sum_{j=1}^{J}a_je^{t\Delta}(x,v_j)$, with the following properties:
	\begin{enumerate}[label=(\roman*)]
		\item \label{i} $\norm{h}_{\infty}=1$
		\item \label{ii} $h(v_j)=\epsilon_j$, for all $v_j \in S$
		\item \label{iii} $\lvert h(x) \rvert < 1$ for all $x \in V \backslash S$.
	\end{enumerate} 
\end{theorem}
\begin{proof} 
	Note:  As defined above, $h|_S=M^{t}a$ and $\norm{h}_{\infty}=1.$	Using Proposition~\ref{inverse}, $M^{t}$ is strictly diagonally dominant and Lemma~\ref{varah} applies, with $a|_S=M^{-t}h|_S$.  Thus, $$\norm{a}_{\infty}=\norm{{M}^{-t}h}_{\ell^{\infty}(S)}\leq \norm{{M}^{-t}}_{{\infty},\infty}\cdot \norm{h}_{\infty}=\norm{{M}^{-t}}_{\infty, \infty}.$$ 
	Next, we extend the heat kernel from $M^{t}$ to $K\left(t,x,y\right)$, where $x,y \in V$.
	We wish to show $\lvert h(v) \rvert < 1$, for every $v \notin S$, knowing we have picked $a_j$'s that make $h(v_j) = \epsilon_j$ for all $v_j \in S.$
	
	So consider an arbitrary  $v \not \in S.$ By the separation of elements in $S$, there exists at most one $v_1 \in S$ such that $\zeta \leq d(v,v_1) < D/2$ and for all other $v_j \in S$ we have $d(v,v_j)\geq D/2.$  Using Folz's bound \cite{Folz} we have
\begin{equation}
	K(t,v,v_1)
	\leq \left(\frac{2et}{\zeta}\right)^{\zeta/2}
	\text{~ and ~~}
		K(t,v,v_j)
		\leq \left(\frac{4et}{D}\right)^{D/4},
	\end{equation}
	and if $v \not \in S$, $t \le T$, 
\begin{equation}
    \begin{split}
    \lvert h(v) \rvert &\leq \norm{a}_{\infty} \left(\left(\frac{2et}{\zeta}\right)^{\zeta/2}
		+ (J-1) \left(\frac{4et}{D}\right)^{D/4} \right)\\
      &\leq \norm{a}_{\infty}\left(\left(\frac{2eT}{\zeta}\right)^{\zeta/2}
		+ (J-1) \left(\frac{4eT}{D}\right)^{D/4} \right)\\
  &< \norm{a}_{\infty} \left(\frac{1}{N}+e^{-T\left(\frac{N}{N-1}\right)\kappa}
		\left(\frac{N-1}{N}\right)-(J-1)\left(\frac{2eT}{D}\right)^{D/2}\right), \text{ from assumption}\\
  &< 1 \text{ from Lemma \ref{varah}}. 
    \end{split}
\end{equation}
\end{proof}

\begin{corollary}
Under the assumptions of the preceding theorem, a solution $g$ to (LP) (see Definition \ref{LP}) is according to Theorem~\ref{thm:dualcert} the unique 
function $g$ such that $f=e^{t\Delta} g$.
\end{corollary}

\section{Recovery with Noisy Measurements} \label{sec:noisyrecovery}

Next, we consider the effect of noise. It is fortuitous that there is no need to alter the construction
of the dual certificate in the treatment of a noisy measurement. We proceed as described in the Introduction
and find a solution to (LP) (as in Definition \ref{LP}). If the noise norm is known, 
so $$
   f = e^{t \Delta} g + \omega
$$
has been observed with $g,\omega$ unknown but $\|\omega\|_2 \le \epsilon$
then the accuracy of the recovered vector is
controlled proportional to $\epsilon$.

We will show that one can stably recover $g$ by finding a $\hat g$ minimizer to (LP) (as in Definition \ref{LP}), which then satisfies
$ \| \hat g - g \|_2 \le C \epsilon$ with some fixed constant $C$.
By our bounds, we know the smallest eigenvalue, $\lambda$, of $M^t$, as in Proposition \ref{inverse}, is bounded by:
$$\lambda \ge \frac{1}{N}+e^{-t\left(\frac{N}{N-1}\right)\kappa}
	\left(\frac{N-1}{N}\right) +(J-1) \left(\frac{2eT}{D}\right)^{D/2} > 0 .$$  
	Consequently, $\|M^{-t}\|\le 1/\lambda$. Because $M^t$ is a contraction, $\lambda <1$, and we know $\|M^{-t}\| \ge 1$. 
We split $\ell^2(V) = \ell^2(S) \oplus \ell^2(S^c)$ and denote the 
corresponding terms $f = f\chi_S+ f\chi_{S^c}$, where the index contains the support of each term.

\begin{lemma}\label{erroronenorm}
	Let $f= e^{t\Delta} g + \omega$ with $\|\omega\|_2\le \epsilon$ and $\hat g$ be a minimizer of (LP) (see Definition \ref{LP}),
	then
	$$
	\| \hat g - g \|_1 \le \sqrt{J} (\lambda^{-1}) (2 \epsilon + \|{\hat g}\chi_{S^c} \|_1) + \|{\hat g}\chi_{S^c} \|_1 \, .
	$$
\end{lemma}
\begin{proof}
	By the splitting and the fact that $S$ contains the support of $g$,
	$$
	\| \hat g - g \|_1 = \|\hat g\chi_S - g\chi_S \|_1 + \|{\hat g}\chi_{S^c}\|_1 \, .
	$$
	We further estimate the first term using the Cauchy-Schwarz inequality and then use the triangle inequality:
	\begin{align}
		\|\hat g\chi_S - g\chi_S \|_1 &\le \sqrt{J} \|\hat g\chi_S - g\chi_S \|_2 \\
		&\le \sqrt{J} \|M^{-t} \| \|M^t (\hat g\chi_S - g\chi_S)\|_2 \\
		&\le \sqrt{J} (\lambda^{-1}) ( \|e^{t \Delta} (\hat g  - g)\|_2 + \| e^{t \Delta} (\hat g  - g) \chi_{S^c}\|_2 \\
		&\le \sqrt{J} (\lambda^{-1}) (2\epsilon + \| \hat g \chi_{S^c} \|_1) \, .
	\end{align}
	The last inequality is due to the tolerance in the feasible region. With both $\hat{g}$ and $g$ feasible and the triangle inequality we have $\|e^{t\Delta}(\hat{g}-g)\|_2 \leq 2\epsilon$. In addition, we have used the contractivity of $e^{t\Delta}$ on $\ell^2$ and the inequality  $\| \hat g\chi_{S^c} \|_2  \le \| \hat g\chi_{S^c} \|_1$. 
\end{proof}

The following lemmata are
similar to results in Cand{\`e}s and Plan \cite{CPlan}. We keep using $h$, the dual certificate, and do not construct an approximate dual certificate yet, even in the noisy case. 

\begin{lemma}\label{lem:offsuppofghat}
	Assume $g\in \Gamma_{J,D}$, $S\in \mathcal{G}_{J,D}$.  Let $f= e^{t\Delta} g + \omega$ with $\|\omega\|_2\le \epsilon$, $S= \text{supp}(g)$ of size $|S|=J$, and $h=e^{t\Delta}a$, a dual certificate for $g$, satisfying
	$ \|h\chi_{S^c}\|_\infty = \xi < 1  $.  If $a$ is chosen such that $\norm{a}_{\infty} \le \lambda^{-1}$, and $\hat g$ is a minimizer of (LP) (see Definition \ref{LP}), then 
	\begin{equation}
		\norm{\hat g\chi_{S^c}}_1\le 2\sqrt{J}\frac{\lambda^{-1}}{1-\xi}\epsilon \, .
	\end{equation}
\end{lemma}
\begin{proof}
	By assumption, $$\norm{h\chi_{S^c}}_{\infty}= \xi<1.$$
	We proceed as follows:
	\begin{align}
		\langle \hat{g}, h\chi_{S^c} \rangle &\le |\langle \hat{g}\chi_{S^c}, h\chi_{S^c} \rangle|\\
		&\le \norm{\hat{g}\chi_{S^c}}_1\norm{h\chi_{S^c}}_{\infty} \,\, ~~ \text{(by H{\"o}lder's inequality)}\\
		&\le \xi\norm{\hat g\chi_{S^c}}_1.
	\end{align}
	Adding $\norm{\hat g\chi_{S^c}}_1 $ to both sides of the resulting inequality, we have:
	$$
	\langle \hat{g}, h\chi_{S^c} \rangle +\norm{\hat g\chi_{S^c}}_1 \le \xi\norm{\hat g\chi_{S^c}}_1 +\norm{\hat g\chi_{S^c}}_1. 
	$$
	So we have the following relation:
	$$(1-\xi)\norm{\hat g\chi_{S^c}}_1\le \norm{\hat g\chi_{S^c}}_1 -\langle \hat{g}, h\chi_{S^c} \rangle.$$
	With $\langle g,h \rangle = \norm{g}_1$ and $h = h\chi_S + h\chi_{S^c}$, we continue by rewriting the right hand side:
	\begin{align}
		(1-\xi)\norm{\hat g\chi_{S^c}}_1 &\le \norm{\hat g\chi_{S^c}}_1 -\langle \hat{g}, h\chi_{S^c} \rangle \\
		&\le \norm{\hat g}_1-\langle \hat{g}, h \rangle-\langle \hat{g}, h\chi_{S^c} \rangle \\
		&= \norm{\hat g}_1-\langle \hat{g}, h \rangle-\langle \hat{g}, h\chi_{S^c} \rangle+\langle g, h\rangle - \norm{g}_1\\
		&\le \langle g-\hat{g}, h \rangle \quad \quad \text{(by assumption, $\norm{\hat g}_1 \le \norm{g}_1$)}\\
		&= \langle g-\hat{g}, e^{t\Delta}a \rangle \\
		&= \langle e^{t\Delta}(g-\hat{g}),a \rangle \\
		&\le \norm{e^{t\Delta}(g-\hat{g})}_2\norm{a}_2 \quad \text{(by Cauchy-Schwarz inequality)} \\
		&\le 2\epsilon \sqrt{J} \norm{a}_{\infty} \quad \text{(using feasibility of $g, \hat{g}$ and support of $a$)} \\
		&\le 2\epsilon \sqrt{J}(\lambda^{-1}).
	\end{align}
	Multiplying by $(1-\xi)^{-1}$ on both sides of the inequality gives the claimed  results.
\end{proof}

We are ready to conclude the error bound for sparse recovery from noisy measurements.
\begin{theorem}\label{thrm: noisyrecovery}Given a finite edge-weighted connected graph $G=(V,E,b)$, 
	a $J$-sparse signal $g$ on $G$, and $f = e^{t \Delta} g + \omega$ with $\|\omega\|_2 \le \epsilon$. 
	Assume $T$, $J$, and $D$ are as in Theorem~\ref{main}, let
	 $\zeta$ be the smallest distance between vertices in $V$ and define $a$ is in Theorem~\ref{main} and
	 $$\xi:=  \norm{a}_{\infty}\left(\left(\frac{2eT}{\zeta}\right)^{\zeta/2}
		+ (J-1) \left(\frac{4eT}{D}\right)^{D/4} \right).$$ If  $\hat g$ is a solution to the linear program (LP) (see Definition \ref{LP}), then
	\begin{equation}
		\| \hat g - g \|_{2} \le 2 \epsilon  (\lambda^{-1}) \left( 1+\frac{\sqrt{J}(1+\lambda^{-1})}{1 - \xi} \right)  \, .
	\end{equation}
\end{theorem}
\begin{proof}
We note by assumption $\xi<1$ and we can apply Lemma \ref{lem:offsuppofghat}. 
	We start with the triangle inequality and use that $g$ restricted to $S^c$ vanishes,
	\begin{equation}
		\| \hat g - g \|_{2} \le \| (\hat g - g) \chi_S \|_2 +  \| \hat g \chi_{S^c} \|_2 \, .
	\end{equation}
	Our estimate for the last term is as derived below using results from Lemma \ref{lem:offsuppofghat}.
	$$
	\| \hat g \chi_{S^c} \|_2 \le \| \hat g \chi_{S^c} \|_1 \le \norm{\hat g\chi_{S^c}}_1\le 2\sqrt{J}\frac{\lambda^{-1}}{1-\xi}\epsilon \, .
	$$
	

	The contribution on $S$ is controlled by
	$$
	\| (\hat g - g) \chi_S \|_2 \le \| (M^t)^{-1} \| \| M^t (\hat g -g )\chi_S \|_2 
	$$
	and then by extending the domain and using the triangle inequality
	$$
	\| (\hat g - g) \chi_S \|_2 \le \| (M^t)^{-1} \| ( \| e^{t  \Delta} (\hat g -g ) \|_2 + \| e^{t  \Delta} (\hat g -g ) \chi_{S^c} \|_2) \, .
	$$
	Next, using the feasibility condition for (LP) (see Definition \ref{LP}) and the contractivity of $e^{t \Delta}$, 
	$$
	\| (\hat g - g) \chi_S \|_2 \le \| (M^t)^{-1}\| (2\epsilon + \| (\hat g -g ) \chi_{S^c} \|_2) \, .
	$$
 The support of $g$ together with the estimate between $\ell^1$ and $\ell^2$ norms
	gives
 $$
	\| (\hat g - g) \chi_S \|_2 \le \| (M^t)^{-1} \| ( 2 \epsilon + \| \hat g   \chi_{S^c} \|_1) \, .
	$$
	By replacing $\|(M^t)^{-1} \|$ by $\lambda^{-1}$, 
	$$
	\| (\hat g - g) \chi_S \|_2 \le (\lambda^{-1}) ( 2 \epsilon + \| \hat g   \chi_{S^c} \|_1) \, .
	$$
	Thus, 
	
	\begin{equation}
		\| \hat g - g \|_{2} \le (\lambda^{-1}) ( 2 \epsilon + 2\sqrt{J}\frac{\lambda^{-1}}{1-\xi}\epsilon) +  2\sqrt{J}\frac{\lambda^{-1}}{1-\xi}\epsilon \le 2 \epsilon  (\lambda^{-1}) \left( 1+\frac{\sqrt{J}(1+\lambda^{-1})}{1 - \xi} \right)  \, .
	\end{equation}
	
\end{proof}

\section{Approximate Dual Certificates and Recovery for More General Signal Models} 
The error bound we have just derived is also useful to control the effect of a change in our signal acquisition model.
Instead of evolving the sparse signal $g$ under the heat semigroup, we can consider measuring a signal consisting of
heat kernels with a spatially varying time parameter.

For the remainder of this section, we consider the problem of recovering a sparse function $g$ when observing
$$
\tilde f(x) = \sum_{y \in V} e^{t_y \Delta} (x,y) g(y) \, ,
$$
where we assume the time parameters $\{t_y\}_{y \in V}$ are known. The simplest approach to this problem is
to instead treat $\tilde f$ as perturbation $\tilde f=e^{T\Delta} g+\omega$ of the noiseless $f= e^{T\Delta} g$ and solve (LP) (as in Definition \ref{LP}). We wish to find out how big the error is that is caused by this model mismatch if all time parameters satisfy $t_y \le T$.

\begin{proposition} \label{prop:modelmismatch}
	With $f$ and $\tilde f$ as above, and $S\in \Gamma_{J,D},$
	$$\norm{f-\tilde{f}}_2\leq\sqrt{J}|e^{T\norm{\Delta}}-1|\norm{g}_2.$$ 
\end{proposition}
\begin{proof}
	Consider the error squared, $$\norm{f-\tilde{f}}_2^2=\sum_{x\in V}\lvert\sum_{y\in S}(e^{T\Delta}(x,y)-e^{t_y\Delta}(x,y))g(y)\rvert^2,$$
	by first looking at
	\begin{equation}
		\lvert\sum_{y\in S}(e^{T\Delta}(x,y)-e^{t_y\Delta}(x,y))g(y)\rvert\leq \left(\sum_{y\in S}|(e^{T\Delta}(x,y)-e^{t_y\Delta}(x,y)|^2\right)^{1/2}(\sum_{y\in S}|g(y)|^2)^{1/2}.
	\end{equation}
	Thus 
	\begin{equation}
		\begin{split}
			\norm{f-\tilde{f}}_2^2&\leq\sum_{x}\sum_{y\in S}|e^{T\Delta}(x,y)-e^{t_y\Delta}(x,y)|^2\norm{g}_2^2\\
			&=\sum_{y\in S}\norm{e^{T\Delta}(\cdot,y)-e^{t_y\Delta}(\cdot,y)}_2^2\norm{g}_2^2\\
			&=\sum_{y\in S}\norm{(e^{T\Delta}-e^{t_y\Delta})\delta_y)}_2^2\norm{g}_2^2\\
			&\leq \sum_{y\in S} \norm{e^{T\norm{\Delta}}-1}^2\norm{g}_2^2\\
			&=J|e^{T\norm{\Delta}}-1|^2\norm{g}_2^2.
		\end{split}
	\end{equation}
	With this being the bound of the error squared, taking the square root yields the desired result.     
\end{proof}

Applying the error bound from noisy recovery, we control the model mismatch.

\begin{corollary}
	Let $\tilde f$ be as above, and $\hat g$ a solution of the minimization problem (LP) (see Definition \ref{LP}), then
	$$
	\| \hat g -g \|_2 \le 2(\lambda^{-1}) \left( \frac{1+\lambda^{-1}}{1 - \xi}\sqrt{J} + 1 \right)|e^{T \|\Delta\|} - 1| \|g\|_2 \, .
	$$
\end{corollary}
\begin{proof}
	Setting $w=f- \tilde f$ with $\|\omega\|_2 \le \epsilon$ and applying the results from Theorem \ref{thrm: noisyrecovery}, we have:
	$$\| \hat g - g \|_{2} \le 2(\lambda^{-1}) \left( \frac{1+\lambda^{-1}}{1 - \xi}\sqrt{J} + 1 \right)\norm{f-\tilde{f}}_2.$$ 
	Lastly, using Proposition \ref{prop:modelmismatch}
	completes our proof.
\end{proof}

 In the work of both Gross \cite{Gross} and Cand{\`e}s and Plan \cite{CPlan} they consider an approximate dual certificate or vector, $\tilde{h}$, with $\norm{\chi_S(\tilde{h}-h)}_2$ small, if $h$ is the exact dual certificate from before.  In this section we use this $\tilde h$ to improve the errors bounds of new type of linear program introduced in the previous section.
\begin{definition}\label{def:approxdual}
    Given $A$ as any $n \times n$ matrix and $g, \in \mathbb{R}^n$, equipped with a dual certificate, $h$.  Define $S:=supp(g) \subset \{1,2,\dots n\}.$  Then a vector, $\tilde{h} \in \mathbb{R}^n$, in the row space of $A$, with the following properties: 
    \begin{enumerate}
        \item $\norm{(\tilde{h} - h)\chi_S}_2 \leq \frac{1}{4}$, and
        \item $\norm{\tilde{h}\chi_{S^c}}_{\infty} \leq \frac{1}{4}$
    \end{enumerate}
    is called an \emph{approximate dual certificate} \cite{CPlan}. 
\end{definition}

\begin{lemma}\label{lemma:atildebounded}
    Given $h=e^{t\Delta}a$, with, $a\in \ell^2(S)$, $a=M^{-t}h$, and $\|a\|_{\infty}< \lambda^{-1} = \|M^{-t} \|$ as before. Let $\tilde{h}:=e^{t\Delta}\tilde{a}$. Assume $\norm{(\tilde{h} - h)\chi_S}_2 \leq \frac{1}{4}$,  then
    $$\norm{a-\tilde a}_2 \leq \frac{1}{4\lambda}.$$
\end{lemma}
\begin{proof}
    On the support of $g$, $\tilde a = (M^t)^{-1}\tilde h \chi_S,$ and as before, $a = (M^t)^{-1} h = (M^t)^{-1}h\chi_S$. Thus, $$
    \norm{(a-\tilde a)\chi_S}_2=\norm{M^{-t}(\tilde h - h )\chi_S}_2\le \norm{M^{-t}}\norm{\tilde{h}\chi_S - h}_2\le \frac{1}{4}(\lambda^{-1}).$$
\end{proof}
We wish to determine the error bound of $\norm{\hat{g}-g}_2,$ if $\hat{g}$ is our best approximation to $g$ by way of an approximate dual.  We begin with an investigation of our best approximate off the support of the true signal.
\begin{lemma}\label{ghatnonsupport}
Let $\hat g$ be a solution to (LP) (see Definition \ref{LP}) and assume an approximate dual certificate exists, then for $\lambda > 1/3$, 
 \begin{equation}
   \|\hat g \chi_{S^c}\|_1 \le   \frac{(8\sqrt{J}+2)}{3\lambda-1} \epsilon \, .
 \end{equation}
 where $\lambda = \| (M^{t})^{-1}\|^{-1}$ as before.
 \end{lemma}
 \begin{proof}
 We assume $h=e^{t\Delta}a $ with $\norm{a}_{\infty}\le \lambda^{-1}$ as before and define $\tilde{h}=e^{t\Delta}\tilde{a}$ with $\tilde{a}$ chosen such that the following conditions hold:
$$
   \| (\tilde h - h) \chi_S \|_2 \le 1/4 
$$
and
$$ \| \tilde h \chi_{S^c} \|_\infty \le 1/4.$$
To introduce $\|\hat{g} \chi_{S^c}\|_2$, first consider:
$$ 
\langle \hat{g},\tilde h \chi_{S^c} \rangle = \sum_{i=1}^N\hat{g}_i(\tilde h \chi_{S^c} )_i = \sum_{j\in S^c}\hat{g}_j\tilde{h}_j.
$$
With this, we can see by H{\"o}lder's inequality that:
$$
\langle \hat g, \tilde h \chi_{S^c} \rangle \le \|\hat{g}\chi_{S^c}\|_1 \|\tilde{h}_{\chi_{S^c}}\|_{\infty} \le \frac{1}{4} \| \hat g \chi_{S^c} \|_1.
$$
Adding $\|\hat{g}_{S^c}\|_1$ to both sides of the inequality along with a simple rewrite, we now have:
$$
\| \hat g \chi_{S^c}\|_1 \le \frac{4}{3} \left(\|\hat g \chi_{S^c}\|_1 - \langle \hat g, \tilde h \chi_{S^c} \rangle\right).
$$

Now using the properties of the dual certificate $h$, we can insert the identity $
  \|g\|_1 =  \| g \chi_S \|_1 = \langle g, h \chi_S \rangle 
$
and get 
$$
   \| \hat g \chi_{S^c}\|_1 \le \frac 4 3 (\|\hat g \chi_{S^c}\|_1 - \| g \|_1 - \langle \hat g, \tilde h \chi_{S^c} \rangle  + \langle g, h \chi_S\rangle) \, . 
$$

We also know 
$$
   \| \hat g \chi_S \|_1 \ge \langle \hat g, h \chi_S \rangle \, , 
$$
then this permits us to extend the domain on the right-hand side
$$
   \| \hat g \chi_{S^c}\|_1 \le \frac 4 3 (\|\hat g \|_1 - \|g\|_1 - \langle \hat g, \tilde h \chi_{S^c}\rangle 
     - \langle \hat g - g , h \chi_S \rangle   ) \, .
$$
Next, we use the Cauchy-Schwarz inequality and condition (1) of Definition \ref{def:approxdual} to replace $h$ by $\tilde h$: $$-\langle \hat{g}-g, h\chi_S \rangle = \langle g-\hat{g}, h\chi_S - \tilde{h}\chi_S\rangle + \langle g-\hat{g}, \tilde{h}\chi_S\rangle\le \frac 1 4 \| (\hat{g}-g)\chi_S \|_2 - \langle \hat{g}-g, \tilde{h}\chi_S\rangle.$$
Thus,
\begin{equation}\label{eq:1bound}
\begin{aligned}
    \| \hat g \chi_{S^c}\|_1 & \le \frac{4}{3} \left(\|\hat g \|_1 - \|g\|_1 - \langle \hat g, \tilde h \chi_{S^c}\rangle - \langle \hat g - g , \tilde h  \chi_S \rangle  + \frac{1}{4} \| (\hat g - g) \chi_S \|_2 \right) \\
    & = \frac{4}{3} \left(\|\hat g \|_1 - \|g\|_1 - \langle \hat g- g , \tilde h \rangle + \frac{1}{4} \| (\hat g - g) \chi_S \|_2 \right) \, .
\end{aligned}
\end{equation}
The last term can be estimated by
\begin{equation}\label{eq:2bound}
\begin{aligned}
    \| (\hat g - g) \chi_S \|_2 &= \| (M^t)^{-1} M^t (\hat g - g) \chi_S \|_2 \\
    &\le \| (M^t)^{-1} \| \| e^{t \Delta} (\hat g - g) \chi_S \|_2 \\
    &\le \| (M^t)^{-1} \| \| e^{t \Delta} \chi_S (\hat g - g) \|_2.
\end{aligned}
\end{equation}
To bound $\| e^{t \Delta} \chi_S (\hat g - g) \|$, we can rewrite the expression and apply the triangle inequality:
\begin{equation}\label{eq:3bound}
\begin{aligned}
    \| e^{t \Delta} \chi_S (\hat g - g) \|_2 & = \| e^{t \Delta}(\hat g - g) - e^{t \Delta} \chi_{S^c} (\hat g - g) \|_2 \\
    & \le \| e^{t \Delta}(\hat g - g) \|_2 + \| e^{t \Delta} \chi_{S^c} (\hat g - g) \|_2 \\
    & \le 2\epsilon + \|\hat{g}\chi_{S^c}\|_2 \quad \text{with } \hat{g},\,g \text{ feasible and the contractivity of } e^{t\Delta},\\
    & \le 2\epsilon + \|\hat{g}\chi_{S^c}\|_1.
\end{aligned}
\end{equation}


Utilizing the inequalities \eqref{eq:1bound}, \eqref{eq:2bound}, and \eqref{eq:3bound}, along with the minimization of the norm ($\norm{\hat g}_1 \le \norm{g}_1$) and subsequent simplification, we obtain the following expression:
 $$
    \| \hat g \chi_{S^c}\|_1 \le \frac 4 3 \langle g - \hat g, \tilde h \rangle  + \frac 2 3 \| (M^t)^{-1} \| \epsilon 
    + \frac 1 3  \| (M^t)^{-1} \| \| \chi_{S^c} \hat g \|_1 \, .
 $$
Using the definition of $\tilde h = e^{t \Delta}\tilde{a}$, we have
 $$
    (1 - \frac 1 3 \| (M^t)^{-1} \| ) \| \hat g \chi_{S^c}\|_1\le \frac 4 3 \langle e^{t\Delta}(g - \hat{g}), \tilde{a} \rangle  + \frac 2 3 \| (M^t)^{-1} \| \epsilon.
 $$
Applying H\"older's inequality, the feasibility of functions $g$ and $\hat{g}$, and the control over the $\ell^2$-norm of $a - \tilde{a}$ (see Lemma \ref{lemma:atildebounded}),  we can bound $\langle g - \hat{g}, \tilde{h} \rangle$:
\begin{align}
\langle g - \hat{g}, \tilde{h} \rangle &\le \langle e^{t\Delta}(g - \hat{g}), \tilde{a} \rangle \\
&\le \| e^{t \Delta} (g - \hat{g}) \|_2 \|\tilde{a}\|_2 \\
&\le (2\epsilon) \|\tilde{a}-a+a\|_2 \\
&\le (2\epsilon) (\|\tilde{a}-a\|_2+\|a\|_2) \\
&\le (2\epsilon)\left(\frac{1}{4}\|M^{-t}\| + \sqrt{J}\|a\|_{\infty}\right).
\end{align}
Replacing the bound for the inner product, $\langle g - \hat{g}, \tilde{h} \rangle$ and using the substitution, $\norm{a}_{\infty}\le\norm{M^{-t}}=\lambda^{-1}$, we now have:
 $$
   (1 - \frac 1 3 (\lambda^{-1}) ) \| \hat g \chi_{S^c}\|_1  
   \le \frac 4 3 (2\epsilon)(\frac 1 4 (\lambda^{-1}) + \sqrt{J}(\lambda^{-1})) + \frac 2 3 (\lambda^{-1}) \epsilon. 
 $$
 Solving, we have 
 $$
   \| \hat g \chi_{S^c}\|_1   \le (1 - \frac 1 3 (\lambda^{-1}) )^{-1}\left(\frac 4 3 (2\epsilon)(\frac 1 4 (\lambda^{-1}) + \sqrt{J}(\lambda^{-1})) + \frac 2 3 (\lambda^{-1}) \epsilon\right).  
 $$
 Simplifying gives the claimed expression.
\end{proof}

\begin{theorem}\label{prop:approxDC} Let $\tilde f$ be as above, assume $S \in \Gamma_{J,D}$ with $\max_{y \in S} t_y \le T$ and
$T,D,$ and $J$ satisfy the inequalities in the preceding section. 
If $\hat g$ is the solution to (LP) (see Definition \ref{LP}) and an approximate dual exists, then for $\lambda > 1/3$, 
	$$
	\| \hat g - g \|_2 \le C\cdot\epsilon,
	$$
with $C=\displaystyle{ 8 \frac{\sqrt{J}(1+\lambda^{-1})+1}{3\lambda-1}}\, . $
	\end{theorem}
 \begin{proof}
     We consider
$$
  \| \hat g - g \|_2 \le \| (\hat g - g ) \chi_S \|_2 + \| \hat g \chi_{S^c} \|_2
   \le \|(M^t)^{-1} \| \| e^{t \Delta} \chi_S (\hat g - g ) \|_2 + \| \hat g \chi_{S^c} \|_2.
   $$ 
Now extending the domain of $\hat g - g$ from $S$ to $V$ gives by triangle inequality
$$
   \| \hat g - g \|_2 \le \|(M^t)^{-1} \| ( \|e^{t \Delta}(\hat g - g) \|_2 + \| e^{t \Delta} \chi_{S^c} \hat g  \|_2) + \| \hat g \chi_{S^c}\|_2
$$ 
and using feasibility, together with the $\ell^2$ norm being bounded by the $\ell^1$ norm,
$$
  \| \hat g - g \|_2 \le \|(M^t)^{-1} \| ( 2 \epsilon + \| \hat g \chi_{S^c} \|_1) +  \| \hat g \chi_{S^c} \|_1 \, .
$$
Lastly, inserting the bound for $\| \hat g \chi_{S^c} \|_1$ from Lemma \ref{ghatnonsupport} and simplifying gives the error bound proportional to $\epsilon$.
 \end{proof}

In the next two sections, we merely show the existence of an approximate dual certificate, thus giving noisy recovery guarantees
as described in the preceding theorem.
 \subsection{Spatially dependent time parameters for heat kernels} \label{tx}

 Again, we wish to consider the modified acquisition with a spatially varying time parameter. 
 In order to improve on the recovery error compared to our earlier treatment in Section 5, we consider a change in the optimization problem.

 \begin{definition}\label{LP'}
     Given $\displaystyle{\tilde f(x)= \sum_{y \in V} e^{t_y \Delta} (x,y) g(y)}$, we say that $\hat g$ is a solution of the linear program (LP$'$) if $\hat g$ minimizes the $\ell^1$-norm subject to 
     $\displaystyle{\hat f = \sum_{y \in V} e^{t_y \Delta} (x,y) \hat g(y)}$ satisfies $\|\tilde f - \hat{f} \|_2 \le \epsilon$.
 \end{definition}
 
 We argue that if $h$ is the dual certificate chosen for (LP) (as in Definition \ref{LP}), according to $h(x) = e^{T\Delta} a$ and $a$ having the same support as 
 $g$, then 
 $$
   \tilde h(x) = e^{t_x \Delta} a(x)
 $$ 
 defines an approximate dual certificate that provides noisy recovery guarantees for (LP$'$) (see Definition \ref{LP'}).

 \begin{lemma}\label{lem:beta}
 Let $g \in \Gamma_{J,D}$ be a function with support $S$, and
 let $T$, $J$ and $D$ be such that 
 $$
    \left( \frac{e T}{2 \zeta} \right)^\zeta + (J-1) \left( \frac{e T}{D} \right)^{D/2} < \frac{\lambda}{4},
 $$
 where $\lambda^{-1} = \norm{M^{-t}}\equiv\|e^{-T\Delta}\vert_{S}\|$, and let $a$ be chosen with support in $S$ such that $h=e^{T\Delta} a, \, h(x) = g(x) /|g(x)|$ for each $x \in S$, and
 $e^{T\|\Delta\|} - 1 \le 1/( 4 \sqrt J )$, then
 $\tilde h$ as given above is an approximate dual certificate.
%
\end{lemma}
\begin{proof}
We wish to show 
$$
   \| (\tilde h - h) \chi_S \|_2 \le \frac 1 4
 $$
 and 
 $$
    \|\tilde h \chi_{S^c}\|_\infty \le \frac 1 4 \, .
 $$

We first note that if $t_x \le T$ for each $x \in V$, then
if $x \in S^c$,
$$
   | e^{t_x \Delta} a(x) | \le \| e^{t_x \Delta} a \chi_{S^c}\|_\infty \le \frac 1 4  \, .
$$ 
What remains is to show the $\ell^2$-norm bound for $h - \tilde h$ restricted to $S$.

To this end, we consider for $x \in S$ 
\begin{align}
  \vert \tilde h(x) - h(x) \vert &= \vert e^{t_x \Delta} a(x) - e^{T\Delta} a(x) \vert \\
   &= \vert \langle (e^{(t_x - T)\Delta} - I) e^{T\Delta} a, \delta_x \rangle \vert \\
    &\le \norm{ e^{(t_x-T) \Delta} - I } \vert \langle e^{T \Delta} a, \delta_x \rangle \vert
   =   \norm{e^{(t_x-T) \Delta} - I }.
\end{align}

When computing the norm of $h - \tilde h$ in $\ell^2(S)$, inserting this gives
\begin{align}
   \| (h - \tilde h)\chi_S \| & \le ( \sum_{x \in S} (\tilde h(x) - h(x))^2 )^{1/2} \\
  & \le (\sum_{x \in S} \| e^{(t_x - T)\Delta}- I \|^2 )^{1/2} \\
 & \le \| (e^{-T\Delta}- I)\vert_S \| \sqrt J\\
 &\le (\lambda^{-1}-1)\sqrt{J}.
\end{align}
Hence, if $T$ is chosen as in our assumption so that $\|h \chi_{S^c} \|_\infty \le 1 /4$ and
$$
   \| e^{T \Delta} - I \| \sqrt J \le 1 /4,
$$ 
then $\tilde h$ defines an approximate dual certificate for recovering $g$ from $\tilde f$.
\end{proof}
\begin{theorem}
    Given an edge-weighted graph, $G=(V,E,b)$, and function $g \in \Gamma_{J,D}$. Let $\hat g$ be the solution to (LP$'$) (see Definition \ref{LP'}), $T$ be the maximum time 
    allowed for a heat kernel associated with the Laplacian $\Delta$ and $J=|S|.$  Assume $T, D, J$, satisfy the following conditions 
 $$
    \left( \frac{e T}{2 \zeta} \right)^\zeta + (J-1) \left( \frac{e T}{D} \right)^{D/2} < \frac{\lambda}{4},
 $$
and
 $$e^{T\|\Delta\|} - 1 \le 1/( 4 \sqrt J ),$$ 
 with $\norm{ e^{-T\Delta}|_S}=\lambda^{-1}, \, \lambda > 1/3$ and $\zeta$ the smallest distance between any vertices in $V$,
then $$\norm{\hat g - g}_2 \le C\epsilon, $$
 with $C=8 \displaystyle{ \frac{\sqrt{J}(1+\lambda^{-1})+1}{3\lambda-1}}\, , $ as shown in Theorem \ref{prop:approxDC}.  
\end{theorem}

\begin{proof}
    With assumptions from Lemma \ref{lem:beta} met, we have shown an approximate dual exists. With the conditions of Theorem \ref{prop:approxDC} met, we have stable recovery, with an error proportional to $\norm{\omega}_2.$ 
\end{proof}

\end{document}